\documentclass[a4paper,11pt]{article}
\usepackage[utf8]{inputenc}
\usepackage[T1]{fontenc}
\usepackage[english]{babel}
\usepackage{amsmath,amssymb,amsthm}
\usepackage[short]{optidef}				
\usepackage[shortlabels]{enumitem} 
\usepackage{comment}
\usepackage{standalone}
\usepackage{mathtools}
\usepackage{fullpage}
\usepackage{float}
\usepackage{mathdots}
\usepackage{tikz}
\usepackage[subrefformat=parens,labelformat=parens]{subcaption}
\graphicspath{{./Figuras/}} 

\newtheorem{theorem}             {Theorem}
\newtheorem{lemma}     	[theorem] {Lemma}        
\newtheorem{conjecture}	[theorem] {Conjecture}

\newtheorem{proposition}[theorem] {Proposition}

\newtheorem{claim}{Claim}

\newcommand{\dw}{\text{dw}}

\newcommand{\authormark}[1]{\textsuperscript{\,#1}}
\newcommand{\showmark}[1]{%
  \hspace*{-1em}\makebox[1em][r]{\authormark{#1}\,}%
  \ignorespaces}

\begin{document}

\title{Domination and packing in graphs}

\author{Renzo Gómez\authormark{1} \and Juan Gutiérrez\authormark{2}\footnote{J. Guti\'{e}rrez is 
supported by Movilizaciones para Investigación AmSud, PLANarity and distance IN Graph theory E070-2021-01-Nro.6997 
and Fondo Semilla UTEC 871075-2022.}} 
\date{}

\maketitle

\begin{center}
\footnotesize

\showmark{1}
Centro de Matemática, Computação e Cognição \\
Universidade Federal do ABC, Santo André, Brazil\\
E-mail: \texttt{gomez.renzo@ufabc.edu.br}

\bigskip
\showmark{2}
Departamento de Ciencia de la Computación\\ 
Universidad de Ingeniería y Tecnología (UTEC), Lima, Perú\\
E-mail: \texttt{jgutierreza@utec.edu.pe}

\end{center}

\begin{abstract}
	Given a graph~$G$, the domination number, denoted by~$\gamma(G)$, is the minimum cardinality of a dominating set in~$G$. 
	Dual to the notion of domination number is the packing number of a graph.
	A packing of~$G$ is a set of vertices whose pairwise distance is at least three. The packing number~$\rho(G)$ of~$G$ is the maximum cardinality of 
	one such set. Furthermore, the inequality~$\rho(G) \leq \gamma(G)$ is well-known. 
	Henning et al.\ conjectured that~$\gamma(G) \leq 2\rho(G)+1$ if~$G$ is subcubic.
	In this paper we progress towards this conjecture by showing that~${\gamma(G) \leq \frac{120}{49}\rho(G)}$ if~$G$ is a bipartite cubic graph. We also show that 
	$\gamma(G) \leq 3\rho(G)$ if~$G$ is a maximal outerplanar graph, and that~$\gamma(G) \leq 2\rho(G)$ if~$G$ is a 
	biconvex graph. Moreover, in the last case, we show that this upper bound is tight.
\end{abstract}

\section{Introduction}	\label{sec:introduction}

For the problem considered here, the input graph is always simple and 
connected (even if this is not stated explicitly). Given a graph $G$,
we say that a set of vertices $X$ of $G$ is a \textit{dominating set}
if every vertex in $V(G)\setminus X$ is adjacent to a vertex in $X$. 
In other words, $X$ intersects the closed neighborhood $N_G[v] = N_G(v) \cup \{ v \}$ 
of each vertex $v \in V(G)$. The \textit{domination number} of~$G$, 
denoted by $\gamma(G)$, is the minimum size of a dominating set of $G$. 

The \textit{distance} between vertices $u$ and $v$ in a graph $G$, 
denoted by $d_G(u, v)$, is the minimum length of a path between them.
Also, we extend this definition to subsets of vertices
of $G$. That is, if $A,B \subseteq V(G)$, 
then~$d_G(A, B)=\min \{d_G(u, v):u \in A, v \in B\}$.
 Dual to the notion of domination number
is the packing number of a graph. A \textit{packing} of a graph $G$ is a set of 
vertices $Y \subseteq V(G)$ such that $N_G[u] \cap N_G[v] = \emptyset$, for every pair 
of distinct vertices $u, v \in Y$. That is, $d_G(u, v) \geq 3$, for every pair of distinct 
vertices $u, v \in Y$. The maximum size of such set is called the \textit{packing number} of $G$, 
and it is denoted by $\rho(G)$. Observe that, while $\gamma(G)$ is the minimum number of closed 
neighborhoods that cover~$G$, $\rho(G)$ is the maximum number of disjoint closed 
neighborhoods in~$G$. Since a dominating set intersects each of these disjoint 
closed neighborhoods, we have that~$\rho(G) \leq \gamma(G)$. 

We observe that, in general, $\rho(G) \neq \gamma(G)$. For instance, 
if we consider $C_4$, the cycle on four vertices, we have 
that $1 = \rho(C_4) < \gamma(C_4) = 2$. Thus, a natural question is to 
investigate how large the ratio $\gamma(G)/\rho(G)$ can be. This is the 
aim of this work, to find upper bounds for~$\gamma(G)$ in terms 
of $\rho(G)$. Burger et al.~{\cite[Lemma~4, Theorem~1]{BurgerHV09}}
observed that, if we consider $G = K_n \square K_n$, the Cartesian product 
of two complete graphs on $n$ vertices, then $\gamma(G) = n$ and $\rho(G) = 1$. 
That is, there exists a graph~$G$ with $\gamma(G) \geq \sqrt{ |V(G)| }\rho(G)$. 
A first upper bound for arbitrary graphs was observed by Henning et 
al.~{\cite[Observation~1]{HenningLR11}}. He showed   
that~$\gamma(G)\leq \Delta(G) \rho(G)$ for any graph $G$. If we restrict 
the class of graphs, we can obtain constant upper bounds. 
L\"{o}wenstein~{\cite[Theorem~8]{LowensteinRR13}} 
showed that $\gamma(G) \leq 2 \rho(G)$ if $G$ is a cactus. Moreover,
this bound is tight (as shown by the cycle $C_4$). 

On the other hand, Meir \& Moon~{\cite[Theorem~7]{MeirM75}} showed 
that $\gamma(G) = \rho(G)$ on trees. Another interesting technique to 
show this equality on some classes of graphs is to use linear programming 
duality~{\cite[Section 7.4]{Scheinerman97}}. To this end, note that $\gamma(G)$ and $\rho(G)$ 
are the optimal values of the following integer programs.
\[
	\setlength{\abovedisplayskip}{0pt}
	\begin{minipage}[t]{0.4\textwidth}
	\begin{mini*}
  	{}
	{\sum_{v \in V(G)}{x_v}}
  	{}
  	{}
  	\addConstraint{x( N[v] ) \geq 1,}{}{\, \forall v \in V(G)}
  	\addConstraint{x_v \in \{0, 1\},}{}{\, \forall v \in V(G).}
	\end{mini*}
	\end{minipage}
\begin{minipage}[t]{0.4\textwidth}
\begin{maxi*}
	{}
	{\sum_{v \in V(G)}{y_v}}
	{}
	{}
  	\addConstraint{y(N[v]) \leq 1,}{}{\, \forall v \in V(G)}	
	\addConstraint{y_v \in \{0, 1\},}{}{\, \forall v \in V(G).}
\end{maxi*}
\end{minipage}
\]
Moreover, consider the fractional analogues of the domination number and packing number 
that are obtained as the linear relaxations of the above integer programs. Note that, if 
the polytope defined by one of these relaxations has integer vertices for a graph $G$, 
then $\gamma(G)=\rho(G)$. A first result using this technique was obtained by 
Farber~{\cite[Corollary~3.4]{Farber84}}. He 
showed that $\gamma(G)=\rho(G)$ for strongly chordal graphs. Brandst\"{a}dt et 
al.~{\cite[Theorem~3.2]{BrandstadtCD98}} extended this result to dually 
chordal graphs. Furthermore, Lov\'{a}sz~\cite{Lovasz75} used the fractional 
packing number to show a ($\log \Delta$)-approximation for the minimum dominating 
set problem, where $\Delta$ denotes the maximum degree of the graph.  

In this work, we are interested in two problems. First, 
we investigate graphs for which $\gamma(G)\leq c \cdot \rho(G)$ 
where $c$ is a constant. The second question of our interest is a conjecture 
raised by Henning et al.~\cite{HenningLR11} for graphs with $\Delta(G) \leq 3$ 
(i.e. subcubic graphs). As mentioned above, they showed that~$\gamma(G)\leq \Delta(G) \rho(G)$ for any graph $G$. 
When $\Delta(G) \leq 2$, they improved this bound 
to $\rho(G) + 1$. For a subcubic graph $G$, they proved 
that $\gamma(G) \leq 2\rho(G)$ if $G$ is also claw-free; and proposed the 
following conjecture.

\begin{conjecture}[Henning et al., 2011]\label{conj:henning-subcubic}
	Let $G$ be a subcubic graph. Then,~ $\gamma(G)\leq 2\rho(G) + 1$, 
	and equality occurs only on three special graphs. 
\end{conjecture}

In this paper we progress towards Conjecture~\ref{conj:henning-subcubic}
by proving the following result.
\begin{itemize}
	\item If $G$ is a bipartite cubic (bicubic) graph , then $\gamma(G)\leq \frac{120}{49}\rho(G)$.
\end{itemize}
As mentioned above, if $G$ is a cactus, then $\gamma(G)\leq 2\rho(G)$. 
A superclass of cacti are outerplanar graphs. Regarding this class, 
we obtain the following result.
\begin{itemize}
	\item If $G$ is a maximal outerplanar graph, then $\gamma(G)\leq 3\rho(G)$.
\end{itemize}
Finally, we consider the class of biconvex graphs and show the following upper bound.
\begin{itemize}
	\item If $G$ is a biconvex graph, then $\gamma(G)\leq 2\rho(G)$. Moreover, 
	this upper bound is tight. 
\end{itemize}

\section{Bicubic graphs} \label{sec:bicubic}

In this section, we show that $\gamma(G) \leq \frac{120}{49}\rho(G)$ if $G$ is a 
bipartite cubic graph, also called a \textit{bicubic} graph. For this, 
we will bound separately the parameters $\gamma(G)$ and $\rho(G)$. 
The following result of Kostochka \& Stocker~\cite{KostochkaS09} gives us 
an upper bound for the domination number of cubic graphs. 
\begin{lemma}[{\cite{KostochkaS09}}]
\label{lemma:domcubickoth}
	For every connected cubic graph $G$ on $n \geq 9$ vertices,
$\gamma(G) \leq \frac{5}{14}n$.
\end{lemma}

Now, we show a lower bound for $\rho(G)$ on bicubic graphs. 
\begin{lemma}\label{lemma:packingcubicn7}
	For every bicubic graph on $n \geq 16$ vertices, $\rho(G) \geq \frac{7}{48}n$.
\end{lemma}
\begin{proof}
	Let $G$ be a bicubic graph with parts $X$ and $Y$. As $G$ is a 
	bipartite regular graph, we have that $|X| = |Y| = \frac{n}{2}$. Let $P$ be a maximum 
	packing that consists of only vertices in~$X$. Let $Q = N(P)$, $R = N(Q)\setminus P$, 
	and $S = N(R) \setminus Q$. Let $p = |P|$.

	Note that $X = P \cup R$. Indeed, if there exists a vertex $v \in X \setminus (P \cup R)$, 
	then $d_G( \{v\}, P ) \geq 4$ (since $R$ contains the vertices at distance two from $P$). 
	Thus $P \cup \{ v \} \subseteq X$ is a packing, which contradicts the maximality of $P$. 
	So, $|R| = |X| - p = \frac{n}{2} - p$. Since $X = P \cup R$, we have that $Y = Q \cup S$.
	As the neighborhood of every vertex in $P$ is pairwise disjoint, we have that $|Q| = 3p$ and $|S| = |Y| - 3p = \frac{n}{2} - 3p$.

	Let $T$ be a maximum packing that consists of vertices in $S$. Let $t =|T|$ and let $W = N(T)$. 
	Note that $W \subseteq R$, and consider the induced bipartite graph $H$ with parts $W$ and $S \setminus T$. 
	We depict this situation in Figure~\ref{fig:bicubicf}. As $G$ is cubic and every vertex $u \in W$ 
	has a neighbor in $Q$ and a neighbor in $T$, every such $u$ has degree at most one in $H$. Observe 
	that every vertex in $S \setminus T$ has at least one neighbor in~$W$. Otherwise, there exists $v \in S$ 
	such that $N_G[v] \cap W = \emptyset$, and the set $T \cup \{v\}$ is a packing contained in $S$, a 
	contradiction to the maximality of $T$. Hence, every vertex in $S\setminus T$ has degree at least 
	one in $H$ and, therefore, $|S \setminus T| \leq |W|$.

	\begin{figure}[H]
		\centering
		\includegraphics[scale=.5]{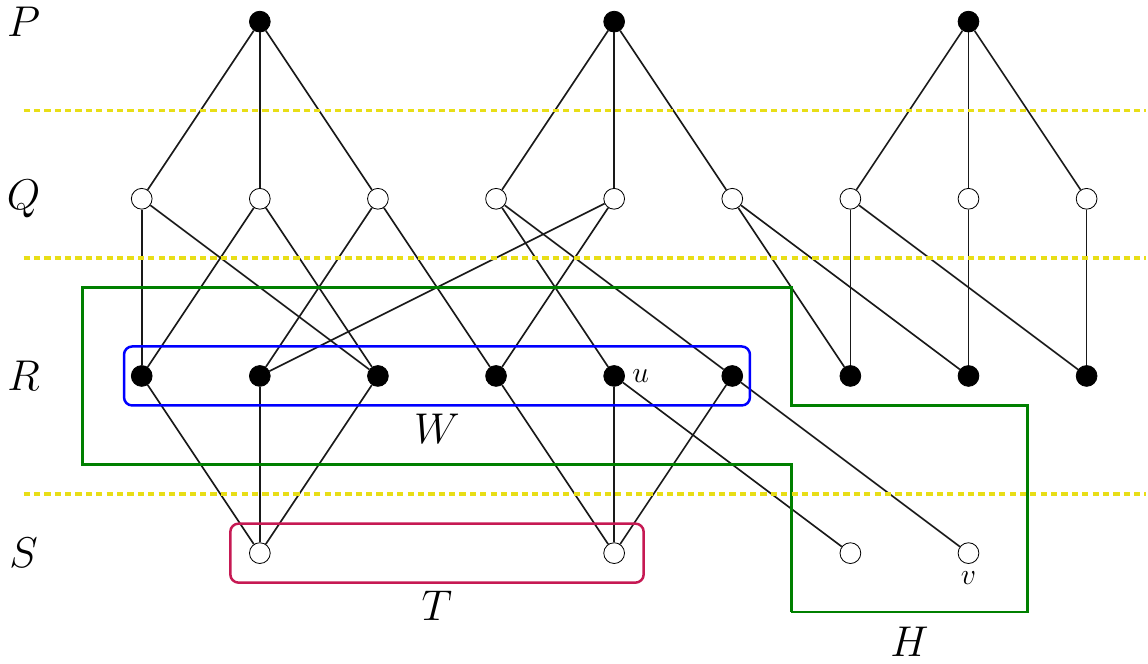}
		\caption{The sets $W$, $T$ and the graph $H$ in the proof of Lemma~\ref{lemma:packingcubicn7}.}
		\label{fig:bicubicf}		
	\end{figure}

	As the neighborhood of every vertex in $T$ is pairwise disjoint, we have that $|W| = 3t$. 
	Thus $|S| \leq 4t$. On the other hand, recall that $|S| = |Y| - 3p$ which implies that $|Y| \leq 3p + 4t$ 
	or, equivalently, that $|Y| \leq 4(p + t) - p$. Observe that $P \cup T$ is a packing of $G$, since the vertices 
	in $P \subseteq X$ are not adjacent to the vertices in $T \subseteq Y$. Therefore,
	\begin{equation}\label{eq:packX}
		|Y| \leq 4\rho(G) - p.
	\end{equation}
	Now, if instead of starting with a maximum packing with vertices only in $X$, we begin with a maximum packing with vertices only in~$Y$, 
	by symmetric arguments, we obtain that 
	\begin{equation}\label{eq:packY}
		|X| \leq 4\rho(G) - p', 
	\end{equation}
	where $p'$ is the cardinality of a maximum packing
	with vertices only in $Y$. In what follows, we show 
	a lower bound for $p$ (resp. $p'$) in terms of the size of $X$ (resp. $Y$). 

	\begin{claim} \label{claim:upperbd}
		Let $G$ be a connected bicubic graph of order $n \geq 16$ with parts $X$ and $Y$.  
		If $P^*$ is a maximum packing with vertices only in $X$ (resp. $Y$), then $|P^* | \geq \frac{|X|}{6}$ 
		(resp. $|P^* | \geq \frac{|Y|}{6}$).
	\end{claim}
	\begin{proof}[Proof of Claim]
		Without loss of generality, suppose that $P^* \subseteq X$. Let $G'$ be the auxiliary graph obtained 
		from $G$ as follows.
		\[
			\begin{array}{rcl}
				V(G') & = & X, \\
				E(G') & = & \{ xx' : x, x' \in X, x \neq x', N_G(x) \cap N_G(x') \neq \emptyset \}.
			\end{array}
		\]
		In other words, $G'$ is the subgraph of $G^2$ (the square of $G$) induced by the vertices of $X$. 
		Observe that, under this setting, a maximum packing with vertices only in $X$ is a maximum independent set of~$G'$. 
		Moreover, since $G$ is cubic, we have that $\Delta(G') \leq 6$. To show the claim, we will use the 
		following result known as \textit{Brook's Theorem}:

		\begin{theorem}[{\cite[Theorem 14.4]{BondyM08}}] \label{thm:brooks}
			If $G$ is a connected graph, and is neither an odd cycle nor a complete graph, 
			then $\chi(G) \leq \Delta(G)$.
		\end{theorem}

		Note that, as $G$ is connected, any path in $G$ between two vertices in $X$ induces 
		a path in $G'$ between those vertices. Therefore, $G'$ is also connected. Furthermore, 
		observe that $G'$ is not a complete graph. Otherwise, since $\Delta(G') \leq 6$, we have that   
		the order $G'$ is at most $7$. But, this implies that $G$ has at most $14$ vertices, a 
		contradiction. 
		
		If $G'$ is an odd cycle, any maximum independent set has size at 
		least $\left \lfloor \frac{|X|}{2} \right \rfloor \geq \frac{|X|}{6}$ and we are done.
		Finally, if $G'$ is not an odd cycle, then,
		by Theorem~\ref{thm:brooks}, we have that $G'$ admits a $\Delta(G')$-coloring. 
		Thus, $G$ has an independent set of size at least $V(G') / \Delta(G') \geq |X| / 6$.
		This finishes the proof of Claim \ref{claim:upperbd}.
\end{proof}

	By Claim~\ref{claim:upperbd}, after adding~(\ref{eq:packX}) and~(\ref{eq:packY}), we obtain the following.
	\[
		\begin{array}{rcl}
			n & \leq & 8\rho(G) - (p + p') \\[2mm]
			\ & \leq & 8\rho(G) - \frac{n}{6}, 
		\end{array}
	\]
	or, equivalently, $\rho(G) \geq \frac{7}{48}n$. 
	This finishes the proof of Lemma \ref{lemma:packingcubicn7}.
\end{proof}

Using an integer programming solver and the database in~\cite{CoolsaetDG23}, we have confirmed 
that the inequality~$\gamma(G) \leq 2\rho(G)$ holds for any bicubic graph of order at most $14$. Therefore,  
Lemma~\ref{lemma:domcubickoth} and Lemma~\ref{lemma:packingcubicn7} imply the following result. 

\begin{theorem}
	If $G$ is a bicubic graph, then $\gamma(G) \leq \frac{120}{49} \rho(G) < 2.45 \rho(G)$.
\end{theorem}

\section{Maximal outerplanar graphs} \label{sec:outerplanar}

A graph is \textit{outerplanar} if it has an embedding in the plane such that 
all of its vertices belong to the boundary of its unbounded face. An outerplanar graph 
is maximal if the addition of a new edge breaks its outerplanarity. In this section, 
we will show that~$\gamma(G) \leq 3\rho(G)$ if $G$ is a maximal
outerplanar graph. In fact, we will show an stronger result.

\begin{theorem}\label{theorem:maxouter}
	For every maximal outerplanar 
	graph~$G$,~${\gamma(G) \leq \min\{3\rho(G), \frac{9}{4}\rho(G)+t/4}\}$, 
	where~$t$ is the number of vertices of degree at most 3 in~$G$.
\end{theorem}

In what follows, given a maximal outerplanar graph $G$, we will 
suppose that $G$ is embedded in the plane (even if this is not stated 
explicitly). Let $G$ be a maximal outerplanar graph. We denote by~$G^{*}$ 
the \textit{dual} of~$G$; that is, the vertices of~$G^{*}$ represent the triangles 
of~$G$, and two vertices in~$G^{*}$ are adjacent if the corresponding triangles 
share an edge. Also, we denote by~$G'$ the~\textit{clique graph} of~$G$; 
that is, the vertices of~$G'$ represent the maximal cliques of~$G$, and two vertices 
in~$G'$ are adjacent if they share a vertex. Since $G$ is maximal outerplanar, 
the maximal cliques of~$G$ are also its triangles and, thus, $G^{*}$ is a 
subgraph of $G'$. To avoid confusion, hereafter, we refer to the elements 
of $V(G^{*})$ (resp. $V(G')$) as the \textit{nodes} of $G^{*}$ (resp. $G'$).
It is known that~$G'$ is a dually chordal graph~{\cite[Corollary~11]{Brandstadt94}}. 
Hence, by Brandstadt et al.~\cite{BrandstadtCD98}, the following property holds.

\begin{proposition}
[{\cite[Theorem 3.2]{BrandstadtCD98}}] \label{prop:gammarho-duallychordal} 
	$\gamma(G')=\rho(G')$.
\end{proposition}

Thus, by Proposition~\ref{prop:gammarho-duallychordal},
it suffices to show that~$\gamma(G) \leq \min \{3 \gamma(G'), \frac{9}{4}\gamma(G')+t/4\}$ 
(see Lemma~\ref{lemma:gammaleq-outerplanar}) and that~$\rho(G) \geq  \rho(G')$ 
(see Lemma~\ref{lemma:rhogeq-outerplanar}) to prove Theorem~\ref{theorem:maxouter}. 
For the first lemma, we will use the following result obtained by Tokunaga~\cite{Tokunaga2013}.

\begin{proposition}
[{\cite[Lemma 1]{Tokunaga2013}}] 
\label{prop:tokunaga4col}
A maximal outerplanar graph can be 4-colored such
that every cycle of length 4 has all four colours.
\end{proposition}

In what follows, if $\{u, v, w \}$ induces a triangle in $G$, we 
denote it by $uvw$. Sometimes we will 
abuse this notation and refer to the node that represents $uvw$ in $G'$ 
also as $uvw$. In each case, its meaning will be clear from the context. 
We now proceed to the proof of the lemma.

\begin{lemma}
\label{lemma:gammaleq-outerplanar}
	For every maximal outerplanar graph~$G$,
	{${\gamma(G) \leq \min\{3\gamma(G'), \frac{9}{4}\gamma(G')+t/4}\}$}, 
	where~$t$ is the number of vertices of degree at most 3 in~$G$.
\end{lemma}

\begin{proof}
Let~$X'$ be a dominating set in~$G'$.
We define a set of vertices~$X$ in~$G$
as the union of all the vertices of the corresponding
nodes in~$G'$. That is,~$X = \bigcup_{abc \in X'}\{a,b,c\}$.
We will show that~$X$ is a dominating set of~$G$.
Let~$u\in V(G)$.
Let~$uvw$ be a triangle in~$G$. As
$X'$ is a dominating set in~$G'$, either~$uvw \in X'$ or $uvw$ 
is adjacent, in $G'$, to a node in $X'$. Therefore, there exists 
a vertex in $\{u, v, w\}$ that belongs to $X$. Thus, 
either~$u \in X$ or~$u$ has a neighbor in~$X$. 
Hence ~$\gamma(G) \leq 3 \gamma(G')$.

Now, consider a $4$-coloring~$\chi=\{C_1,C_2,C_3,C_4\}$ of $G$ as in Proposition \ref{prop:tokunaga4col}.
We will show that
$C_{123} = (C_1 \cup C_2 \cup C_3) \cap X$ 
dominates every vertex of degree at least 4.
Let~$u$ be a vertex of degree at least 4. Since 
the neighborhood of a vertex in an outerplanar 
graph induces a path~{\cite[Theorem~1A]{Allgeier09}}, 
let~$v_1v_2 \ldots v_\ell$, $\ell \geq 4$, be a path in~$G$ such that
every~$v_i$ is a neighbor of~$u$.
Suppose by contradiction that~$C_{123}$ does not dominate~$u$. 
Then~$X \cap \{v_1, v_2, \ldots, v_\ell \} \subseteq C_4$. 
Furthermore, $X \cap \{ v_1, v_2, \ldots, v_\ell \} \neq \emptyset$ 
since $X$ is a dominating set of $G$. 
Let~$v_i \in X \cap \{v_1,v_2, \ldots, v_\ell \}$.
As the edge~$uv_i$ is in at most two triangles, there exists
a triangle~$uv_jv_k$ that do not contain~$v_i$ and $\{u,v_i,v_j,v_k\}$ form a cycle. 
But then~$u, v_j, v_k \notin C_4$. This implies that $u, v_j, v_k \notin X$ and, 
therefore, the node~$uv_jv_k$ is not dominated by~$X'$ in~$G'$, a contradiction. 
Hence,~$C_{123}$ dominates every vertex of degree at least~$4$. 

Thus,
$C_{123} \cup U_{123}$ is a dominating set of~$G$,
where~$U_{123}$ is the set of vertices not dominated by~$C_{123}$.
By repeating the same argument, we conclude that
$C_{ijk} \cup U_{ijk}$ is a dominating set for any~$ijk \in \{123,124,134,234\}$,
where $U_{ijk}$ is the set of vertices not dominated by
$C_{ijk}$.
Now, note that~$U_{ijk} \cap U_{ijk'} = \emptyset$,
for every~$i,j,k,k' \in  \{1,2,3,4\}$, $k \neq k'$. 
Indeed, suppose by contradiction that there exists a vertex~$u \in U_{ijk} \cap U_{ijk'}$.
As~$\{i,j,k,k'\} = \{1,2,3,4\}$, no neighbor of~$u$ is in~$X$, a contradiction to the fact 
that~$X$ dominates~$G$. Finally, if we sum up $|C_{ijk}|$, for each $ijk \in \{123,124,134,234\}$, 
we obtain $3|X|$ as each color $C_i$ is repeated $3$ times in that summation. Hence, by an averaging 
argument, we have that~$\gamma(G) \leq \frac{3|X| + t}{4} = \frac{9|X'| + t}{4}$.
\end{proof}

We now proceed to the proof of our second lemma.

\begin{lemma}\label{lemma:rhogeq-outerplanar}
For every maximal outerplanar graph~$G$,
we have~$\rho(G) \geq \rho(G')$.
\end{lemma}
\begin{proof}
First, note that~$G^{*}$ is a spanning
tree of~$G'$. Let~$Z$ be a packing in~$G'$.
We select an arbitrary node in~$Z$, say~$r$,
and direct the edges of~$G^{*}$ away from $r$, 
obtaining a rooted directed tree~$T$ with root~$r$.
For any node~$u$, we denote by~$T_u$ the subtree
of~$T$ rooted at~$u$, and by~$G_u$, the corresponding
induced subgraph in~$G$, that is,
$G_u=G[\{v:v \text{ is a vertex of some node in }T_u\}]$.
For a node~$u \in V(T)$, we denote by~$p(u)$ the 
parent of~$u$ in~$T$, and by~$h(u)$, the height
of~$u$ in~$T$ (see Figure~\ref{fig:outerplanar}). Moreover, 
since the nodes of $G'$ represent sets of vertices (triangles) in $G$, 
if $u, v \in V(G')$, we denote by $u \cap v$ the intersection of these sets. 
In a similar way, $d_G(u, v)$ denotes the distance between 
these sets in $G$.

\begin{figure}[H]
	\centering
	\includegraphics[scale=.5]{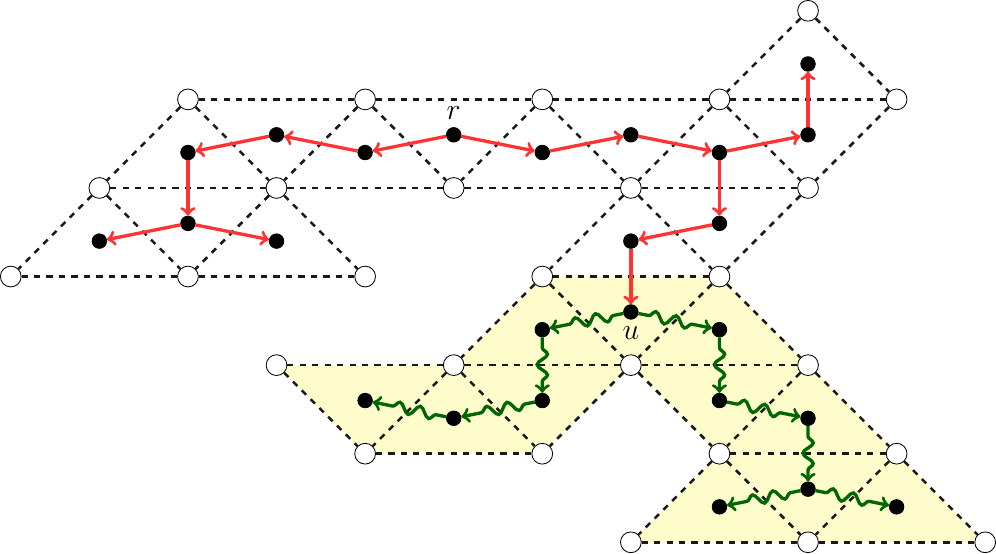}
	\caption{A maximal outerplanar graph~$G$ is represented by the white circles and by the dashed edges. The nodes of~$G^{*}$ are the black circles. 
	By selecting a node~$r$, we obtain the directed tree~$T$. The arcs with curvilinear shape are in~$T_u$, and~$G_u$ is the induced subgraph obtained from the union of the shaded triangles. In this case~$h(u) = 4$.}
	\label{fig:outerplanar}
\end{figure}

Observe that, it suffices to show that, for any~$u \in Z$, 
we have a packing of size~$|Z \cap T_u|$ in~$G_u$. Indeed,
if that is the case, then, for~$u=r$, we have
in~$G$ a packing of size~$|Z\cap T|=|Z|$. We will show this property by 
induction on~$h(u)$. Actually, we will prove a stronger statement.

\begin{claim}\label{claim:YinGusizeZcapTu}
For any~$u \in Z$, there exists a packing~$Y$
in~$G_u$ of size~$|Z \cap T_u|$ such that the vertices
in~$u \cap p(u)$ are not in~$Y$.
\end{claim}
\begin{proof}[Proof of Claim]
By induction on~$h(u)$.
If~$h(u)=0$, then~$u$ is a leave in~$T$, which implies that the only neighbor 
of~$u$ in~$T$ is its parent~$p(u)$.
Hence, if~$u=abc$ and~$u \cap p(u)=\{a,b\}$, then~$Y=\{c\}$ 
and the claim follows. Now, suppose that~$h(u)>0$. Let~$c$ be the vertex in~$u$
that is not in~$u\cap p(u)$. As~$h(u)>0$,~$T_u$
is not a single node.
Let~$Q$ be the set of nodes $q \in (Z \cap T_u) \setminus \{ u \}$ 
such that the path in~$T$ from~$u$ to~$q$ does
not contain any other node in~$(Z \cap T_u) \setminus \{ u \}$.
By the induction hypothesis, for every~$q \in Q$,
there exists a packing~$Y_q$ in~$G_q$ of size
$|Z \cap T_q|$ that does not contain any vertex
in~$q \cap p(q)$.

Let~$Y=\{c\} \cup \bigcup_{q \in Q}Y_q$.
We will show that~$Y$ is a packing in~$G$. 
By the definition of~$Q$, if $q, q'$ are distinct nodes of $Q$, 
then~$G_q \cap G_{q'} = \emptyset$.
Thus, the vertices in~$q \cap p(q)$ separate $Y_q$ 
from~$Y_{q'}$ in~$G$. As~$Y_q$ does not contain any vertex 
in~$q \cap p(q)$, we have that~$d_G(Y_q, Y_{q'}) \geq 3$.
Also, as~$u,q \in Z$, we have~$d_{G'}(u,q) \geq 3$.
This implies that~$d_G( u , q ) \geq 2$.
As the vertices in~$q \cap p(q)$ separate~$Y_q$ from~$c$ in $G$ and, 
as~$Y_q$ does not contain any vertex in $q \cap p(q)$, we have that~$d_G( c, Y_q) \geq 3$.
We conclude that~$Y$ is a packing in~$G$ with the desired properties.
\end{proof}

By Claim~\ref{claim:YinGusizeZcapTu}, with~$u = r$, we have in~$G$ a packing of size~$|Z|$. Hence,
$\rho(G) \geq |Z| = \rho(G')$, as we wanted.
\end{proof}



\section{Biconvex graphs} \label{sec:biconvex}

In this section, we show that~$\gamma(G) \leq 2\rho(G)$ on 
biconvex graphs. Also, we exhibit a family of biconvex graphs 
that attain this bound. Let~$G$ be a connected bipartite graph 
with parts~$X$ and~$Y$. We say that~$X$ has a \textit{convex ordering} 
if there exists an ordering of its vertices, say~$\langle x_1, \ldots, x_n \rangle$, such 
that~$N_G(y)$ consists of consecutive vertices in that ordering, for 
each~$y \in Y$. Furthermore, we say that~$G$ is \textit{biconvex} if 
both~$X$ and~$Y$ admit a convex ordering. A subclass of biconvex graphs 
that will be of our interest are the bipartite permutation graphs. 
Yu \& Chen~\cite{YuC95} showed that, given a biconvex graph~$G$ with parts~$X$ 
and~$Y$, we can obtain a bipartite permutation graph from~$G$ after removing 
some (maybe none) of the first and last vertices of a convex ordering of~$X$. 
We will use this result along with a structural result regarding 
bipartite permutation graphs to show the desired upper bound on~$\gamma(G)$. 
In what follows, we present these results and the necessary definitions. 

Let~$G$ be a connected biconvex graph with parts~$X$ and~$Y$.
Also, let~$\langle x_1, \ldots, x_n \rangle$ and~$\langle y_1, \ldots, y_m \rangle$ be convex 
orderings of~$X$ and~$Y$, respectively. Let~$x_L$ (resp.~$x_R$) be the 
vertex in~$N(y_1)$ (resp.~$N(y_m)$) such that~$N(x_L)$ (resp.~$N(x_R)$) is not 
properly contained in any other neighborhood set. In case of ties, we choose 
as~$x_L$ (resp.~$x_R$)  the smallest (resp.~largest) vertex that satisfies the 
previous condition. Without loss of generality, suppose that~$x_L \leq x_R$, 
otherwise consider the reverse ordering of~$X$. Let~$G_p = G[X_p \cup Y]$ where 
$X_p = \{ x_i : x_L \leq x \leq x_R\}$. Yu \& Chen~\cite{YuC95} showed the 
following (see also~{\cite[Lemma~9]{AbbasL00}}).

\begin{lemma}[{\cite[Lemma 7]{YuC95}}] \label{lem:yu}
	$G_p$ is a connected bipartite permutation graph.
\end{lemma}

Consider the convex orderings~$\langle x_1, \ldots, x_n \rangle$ and~$\langle y_1, \ldots, y_m \rangle$. 
These orderings form a \textit{strong} ordering of~$V(G)$ if~$x_iy_a, x_ky_c \in E(G)$ 
whenever~$x_iy_c, x_ky_a \in E(G)$, for~$i \leq k$ and~$a \leq c$. 
Abbas \& Stewart~\cite{AbbasL00} proved the following properties of~$X - X_p$.
	
\begin{lemma}[{\cite[Lemma 10]{AbbasL00}}] \label{lem:lorna}
	Let~$G$ be a biconvex graph with parts~$X$ and~$Y$. There exist 
	convex orderings~$\langle x_1, \ldots, x_n \rangle$ 
	and~$\langle y_1, \ldots, y_m \rangle$ that satisfy the following properties:
	\begin{enumerate}[$a)$]
		\item $G_p$ is a connected bipartite permutation graph.
		\item $X_p = \langle x_L, x_{L+1}, \ldots, x_R \rangle$ and~$Y = \langle y_1, y_2, \ldots, y_m \rangle$ 
		is a strong ordering of~$V(G_p)$.
		\item For all~$x_i$ and~$x_j$, where~$x_1 \leq x_i < x_j \leq x_L$ 
		or~$x_R \leq x_j < x_i \leq x_n$, we have that~$N(x_i) \subseteq N(x_j)$. \label{it:important} 
	\end{enumerate}
\end{lemma}

Now, we describe a decomposition of a bipartite permutation graph that along 
with Lemma~\ref{lem:lorna} give us the necessary tools to show the main 
result of this section. Let~$H$ be a bipartite permutation graph 
with convex orderings~$A = \langle a_1, \ldots, a_n \rangle$ and~$B = \langle b_1, \ldots, b_m \rangle$. 
Uehara \& Valiente~\cite{UeharaV07} showed that~$H$ can be decomposed into 
subgraphs~$(K_1 \cup J_1), \ldots, (K_k \cup J_k)$, called a \textit{complete bipartite decomposition} 
of~$H$, in the following way. Let~$K_1$ be the graph induced by~$N_H(a_1) \cup N_H(b_1)$. Let~$J_1$ be the 
set of isolated vertices in~$H - K_1$ (maybe empty). Next, consider the 
graph~$H' = H - (K_1 \cup J_1)$ and repeat this process until the graph becomes empty.
For~$i = 1, \ldots, k$, let~$\ell_A(K_i)$ and~$r_A(K_i)$ be the vertices 
in~$A \cap V(K_i)$ such that~$\ell_A(K_i)\leq a \leq r_A(K_i)$, for each vertex~$a \in A \cap V(K_i)$. 
Analogously, we define~$\ell_B(K_i)$ and~$r_B(K_i)$. We say that subgraphs~$H_1$ and~$H_2$ of~$H$ 
are adjacent if there exists an edge, in~$H$, with one end in~$H_1$ and the other end in~$H_2$.
Uehara \& Valiente~{\cite[Lemma~1, Lemma~2, Theorem~5]{UeharaV07}} showed that a complete bipartite decomposition of~$H$ 
satisfies the following properties.

\begin{enumerate}[$i)$]
	\item $K_i$ is a complete bipartite graph (with at least one edge); and~$K_i$ 
	is adjacent to~$K_{j}$ if only if~$|i - j| = 1$. \label{it:complete}
	\item Either~$J_i \subseteq A$ or~$J_i \subseteq B$; moreover~$N_H(J_i) \subseteq V(K_i)$, 
	and if~$a, a' \in J_i$ and~$a \leq a'$, then~$N_H(a') \subseteq N_H(a)$. \label{it:independent}
	\item If~$J_i \subseteq B$ (resp.~$J_i \subseteq A$), then~$r_A(K_i) \in N_H( z )$ 
	(resp.~$r_Y(K_i) \in N_H(z)$), for every vertex~$z \in J_i$. \label{it:ji} 
\end{enumerate}

\begin{figure}
	\centering
	\includegraphics[scale=.5]{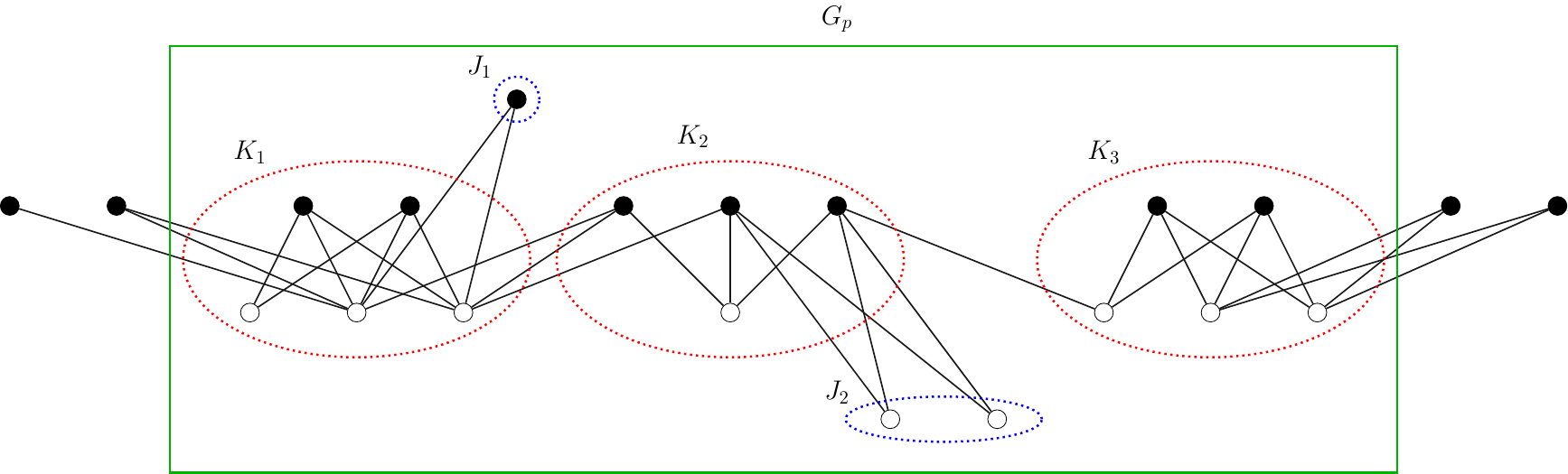}
	\caption{The structure of a biconvex graph~$G$.}
	\label{fig:biconvex}
\end{figure}

Let~$G$ be a connected biconvex graph with parts~$X$ and~$Y$, and let~$G_p$ the bipartite 
permutation graph given by Lemma~\ref{lem:lorna}. Furthermore, let~$(K_1 \cup J_1), \ldots, (K_k \cup J_k)$ 
be a complete bipartite decomposition of~$G_p$. In what follows, we call~$k$ 
the \textit{decomposition width} of~$G$, and denote it by~$\dw(G)$. 
Figure~\ref{fig:biconvex} depicts the structure of~$G$. The following two results 
relate~$\gamma(G)$ and~$\rho(G)$ to~$\dw(G)$. Throughout the proofs of these results, 
we use properties~\ref{it:complete}, \ref{it:independent} and~\ref{it:ji}.

\begin{lemma} \label{lem:bi-pack}
	Let~$G$ be a connected biconvex graph with parts~$X$ and~$Y$ 
	such that~$\dw(G) = k$. Then,~$\rho(G) \geq k$. Furthermore, 
	if~$\max\{ d_G(x_1, a) : a \in J_1 \} \geq 3$ or~$N_G(x_n) \subseteq J_k$, 
	then~$\rho(G) \geq k + 1$.
\end{lemma}
\begin{proof}
	Let~$(K_1 \cup J_1), \ldots, (K_k \cup J_k)$ be a complete bipartite decomposition of~$G_p$. 
	First, we will show that~$\rho(G) \geq k$. Let 
	\[
		S := \{ \ell_{X}(K_{i}) : 1 \leq i \leq k, \, i \text{ is odd} \} \cup \{ \ell_{Y}(K_{i}) : 1 \leq i \leq k, \, i \text{ is even} \}.
	\]
	Note that~$|S| = k$. Now, we show that~$S$ is a packing of~$G$. By definition of~$K_i$, we have 
	that~$N_G( \ell_X(K_i) ) \cap V(K_{i + 1}) = \emptyset$. Thus~\ref{it:complete} implies that the 
	distance between~$\ell_X(K_i)$ and~$\ell_X(K_j)$ is at least 4, if~$i \neq j$ and~$i, j$ are odd. 
	The same arguments applies to~$\ell_Y(K_i)$ and~$\ell_Y(K_j)$ when~$i \neq j$ and~$i, j$ are even.
	Finally, by~\ref{it:complete}, the only case left to consider is a pair of vertices~$\ell_X(K_i)$ 
	and~$\ell_Y(K_{i + 1})$. Since~$N_G( \ell_X(K_i) ) \cap V(K_{i + 1}) = \emptyset$, the distance 
	between~$\ell_X(K_i)$ and~$\ell_Y(K_{i+1})$ is at least 3. Therefore,~$\rho(G) \geq k$. 

	Now, suppose that~$N_G(x_n) \subseteq J_k$. By~\ref{it:independent}, the set~$J_k \subseteq Y$, 
	and by the definition of~$K_k$, we have that~$N_G(\ell_X(K_k)) \cap J_k = \emptyset$. This 
	implies that the distance between~$x_n$ and the set~$\{ \ell_X(K_k), \ell_Y(K_k) \}$ is at least 3. 
	Thus, if we consider the set~$S$ defined before, then~$S' = S \cup \{ x_n \}$ is a packing of~$G$. 
	Therefore~$\rho(G) \geq k + 1$.

	Finally, suppose that~$\max\{ d_G(x_1, a) : a \in J_1 \} \geq 3$. We distinguish two cases. 
	
	\vspace{2mm}
	\noindent \textbf{Case 1:} $J_1 \subseteq X$.

	Let~$x' \in J_1$ such that~$d_G(x_1, x') \geq 3$, and consider the following set:
	\[
		S := \{ \ell_X(K_i) : 2 \leq i \leq k, i \text{ is odd} \} \cup \{ \ell_Y(K_i) : 2 \leq i \leq k, i \text{ is even}  \}.
	\]
	Then,~$|S| = k - 1$. By similar arguments as before,~$S$ is a packing of~$G$. Now, 
	consider the set~$S' = S \cup \{ x_1, x' \}$. Since~$N_G(x_1) \cup N_G(x') \subseteq V(K_1)$, we have 
	that the distance between~$\ell_Y(K_2)$ and~$\{ x_1, x' \}$ is at least 3 and, thus~$S'$ is 
	a packing of~$G$. 

	\vspace{2mm}
	\noindent \textbf{Case 2:} $J_1 \subseteq Y$.

	In this case, consider the set 
	\[
		S := \{ \ell_X(K_i) : 2 \leq i \leq k, i \text{ is even} \} \cup \{ \ell_Y(K_i) : 2 \leq i \leq k, i \text{ is odd} \}.
	\]
	As in the previous case~$S$ is a packing of~$G$ such that~$|S| = k - 1$. Consider the set~$S' = S \cup \{ x_1, y' \}$ 
	where~$d_G(x_1, y') \geq 3$,~$y' \in J_1$. As~$N_G(y') \subseteq V(K_1)$, we have that~$d_G(y', \ell_X(K_2)) \geq 3$. 
	Finally, we will show that~$N_G(x_1) \cap N_G(\ell_X(K_2)) = \emptyset$. Suppose by contradiction that there exists 
	a vertex~$y^* \in N_G(x_1) \cap N_G(\ell_X(K_2))$. 
	Since~$N_G(x_1) \subseteq V(K_1)$, we have that~$y^* < y' < \ell_Y(K_2)$. As~$G$ is biconvex, this implies 
	that~$\ell_X(K_2)$ is adjacent to~$y' \in J_1$. This contradicts the fact that~$J_1$ is a set of isolated 
	vertices in~$G_p - K_1$. Thus~$S'$ is a packing of~$G$.
	
	\vspace{2mm}
	\noindent Therefore~$\rho(G) \geq k + 1$. 
\end{proof}

\begin{lemma} \label{lem:bi-dom}
	Let~$G$ be a connected biconvex graph with parts~$X$ and~$Y$ such that~$\dw(G) = k \geq 2$. 
	If~$\max\{ d_G(x_1, a) : a \in J_1 \} \leq 2$ and~$N_G(x_n) \cap V(K_k) \neq \emptyset$, 
	then~$\gamma(G) \leq 2k$; otherwise we have that~$\gamma(G) \leq 2k + 2$.
\end{lemma}
\begin{proof}
	Let~$(K_1 \cup J_1), \ldots, (K_k \cup J_k)$ be a complete bipartite decomposition of~$G_p$.
	First, suppose that~$\max\{ d_G(x_1, a) : a \in J_1 \} \leq 2$ and~$N_G(x_n) \cap V(K_k) \neq \emptyset$. 
	In this case, we claim that~$J_1 \subseteq X$. Otherwise~$J_1 \subseteq Y$, and 
	since~$N_G(x_1) \subseteq N_G(x_L) = N_G(\ell_X(K_1))$, we have that~$d_G(x_1, J_1) \geq 3$, 
	a contradiction. Thus~$J_1 \subseteq X$. Let 
	\[
		S = \{ r_X(K_i), r_Y(K_i) : 2 \leq i \leq k - 1\}. 
	\]
	By~\ref{it:ji}, the set~$S$ dominates~$K_i \cup J_i$, for~$i = 2, \ldots, k - 1$. 
	Since~$d_G(x_1, x') \leq 2$, for each~$x' \in J_1$, \ref{it:independent} implies 
	that there exists~$y^* \in N_G(x_1) \cap N_G(x')$, for every~$x' \in J_1$. 
	Thus, the set~$\{ r_X(K_1), y^* \}$ dominates~$K_1 \cup J_1 \cup \{ x_1, \ldots, x_L \}$. 
	On the other hand, since~$N_G(x_n) \cap V(K_k) \neq \emptyset$, there 
	exists~$y' \in N_G(x_n) \cap V(K_k)$. By similar arguments as before, the set~$\{y', r_X(K_k)\}$ 
	dominates~$K_k \cup J_k \cup \{ x_R, \ldots, x_n \}$. Observe that, in case~$J_k \subseteq X$, 
	the vertex~$y'$ dominates~$J_k$ since~$G$ is biconvex. Therefore,~$S \cup \{ r_X(K_1), y^*, r_X(K_k, y') \}$ 
	dominates~$G$ implying that~$\gamma(G) \leq 2k$. 

	Finally, suppose that~$\max\{ d_G(x_1, a) : a \in J_1 \} \geq 3$ or~$N_G(x_n) \cap V(K_k) = \emptyset$. 
	First, let~$S$ be the same set defined above. Let~$y^*$ and~$y'$ be vertices adjacent to~$x_1$ and~$x_n$, 
	respectively. By Lemma~\ref{lem:lorna}~$c)$, the set~$\{ y^*, y'\}$ dominates~$\{x_1, \ldots, x_L, x_R, \ldots, x_n \}$. 
	Therefore, the set 
	\[
		S' = S \cup \{ r_X(K_1), r_Y(K_1), r_X(K_k), r_Y(K_k), y^*, y' \}
	\]
	dominates~$G$, which implies that~$\gamma(G) \leq 2k + 2$. 
\end{proof}

\begin{theorem}
	If~$G$ is a connected biconvex graph, then~$\gamma(G) \leq 2\rho(G)$.
\end{theorem}
\begin{proof}
	Let~$k = \dw(G)$. Lemma~\ref{lem:bi-pack} and Lemma~\ref{lem:bi-dom} imply the 
	result when~$k \geq 2$. So suppose that~$k = 1$, that is~$G_p = K_1 \cup J_1$. 
	We distinguish two cases.

	\vspace{2mm}
	\noindent \textbf{Case 1:} $d_G(x_1, x_n) = 2$.

	In this case, there exists a vertex~$y^* \in N_G(x_1) \cap N_G(x_n)$. Since 
	$G$ is biconvex, the vertex~$y^*$ dominates~$X$. Moreover, as~$k = 1$, we 
	have that~$N_G( \ell_X(K_1) ) = Y$. Thus, the set~$\{ \ell_X(K_1), y^* \}$ 
	dominates~$G$. Therefore~$\gamma(G) \leq 2 \leq 2\rho(G)$.

	\vspace{2mm}
	\noindent \textbf{Case 2:}~$d_G(x_1, x_n) \geq 3$.

	In this case, $\{ x_1, x_n \}$ is a packing of~$G$,  which implies that~$\rho(G) \geq 2$. 
	Let~$y^* \in N_G(x_1)$ and~$y' \in N_G(x_n)$. By Lemma~\ref{lem:lorna}~$c)$, 
	we have that the vertex~$y^*$ dominates~$\{ x_1, \ldots, x_L \}$ and~$y'$ 
	dominates~$\{ x_R, \ldots, x_n \}$. Furthermore, the set~$\{ r_X(K_1), r_Y(K_1) \}$ 
	dominates~$K_1 \cup J_1$. Therefore, $\gamma(G) \leq 4 \leq 2\rho(G)$. 
\end{proof}

We conclude this section showing that the previous upper bound is tight. Let~$k$ be a positive 
integer. We define the graph~$G'_k$ as follows. Let~$G_i$ be a copy of the complete bipartite 
graph~$K_{2,2}$ such that~$V(G_i) = \{ x_{2i - 1}, x_{2_i} \} \cup \{ y_{2i - 1}, y_{2_i} \}$, 
for~$i = 1, \ldots, k$. Then  
\[
	\begin{array}{rcl}
		V(G'_k) & = & \displaystyle \bigcup_{i=1}^{k}{V(G_i)}, \\[2mm]
		E(G'_k) & = & \displaystyle \bigcup_{i=1}^{k}{E(G_i)} \cup \{ x_{2i}y_{2i + 1} : 1 \leq i \leq k - 1\}. 
	\end{array}
\]
An example of~$G'_k$ for~$k = 3$ is given in Figure~\ref{fig:tight}. Since we need at least two vertices to 
dominate each~$G_i$, we have that~$\gamma(G'_k) = 2k$. On the other hand, the set 
\[
	S = \{ x_{2i - 1} : 1 \leq i \leq k, i \text{ is odd} \} \cup  \{ y_{2i - 1} : 1 \leq i \leq k, i \text{ is even} \}
\]
is a packing of~$G'_k$. Moreover, any packing of~$G$ does not contain more than one vertex from each~$G_i$. 
Therefore~$\rho(G'_k) = k$ which implies that~$\gamma(G'_k) = 2\rho(G'_k)$. Finally, we observe 
that~$G'_k$ is a bipartite permutation graph. Therefore this upper bound is also tight if we restrict 
to this class.

\begin{figure}
	\centering
	\includegraphics[scale=.5]{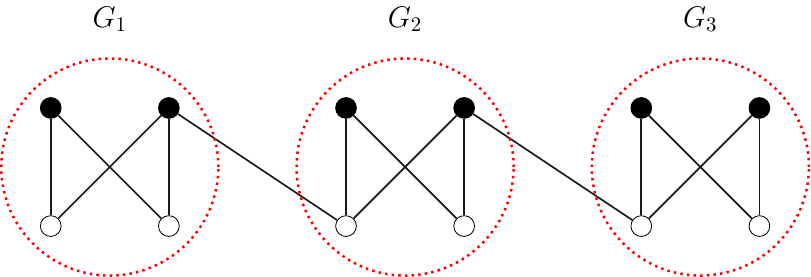}
	\caption{The graph~$G'_3$.}
	\label{fig:tight}
\end{figure}

\section{Final remarks}

In this work, we progress towards Conjecture 1 by showing 
that~$\gamma(G) < 2.45\rho(G)$ for every bicubic graph~$G$.
We believe that this result can be extended to 
bipartite and subcubic graphs by using analogous techniques. For this, 
a first step could be to extend the result obtained by 
Kostochka \& Stocker~\cite{KostochkaS09} to subcubic graphs. 

Henning et al.~\cite{HenningLR11} observed that~$\gamma(G) \leq \Delta \rho(G)$ 
for any graph~$G$ without isolated vertices. We propose a stronger version of 
Conjecture~1 as follows.
\setcounter{theorem}{1}
\begin{conjecture}\label{conj:Delta-1}
	For any connected graph~$G$ with at least two vertices, 
$\gamma(G) \leq (\Delta-1) \rho(G)+1$.
\end{conjecture}
We believe our techniques are still far from settling Conjecture \ref{conj:Delta-1}. 
However, we can relax this problem by considering a decrease 
in the second parameter of the upper bound, that is, is it true 
that,~$\gamma(G) \leq \max\{(\Delta-1) \rho(G),\Delta (\rho(G)-1)\}+1$, 
for every connected graph~$G$ with at least two vertices?

\begin{figure}[htb]
	\centering
	\includegraphics[scale=.5]{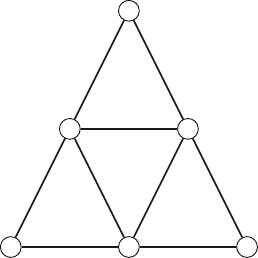}
	\caption{The sun graph satisfies that $\gamma(G) = 2\rho(G)$.}
	\label{fig:sun}	
\end{figure}

On the other hand, we also investigated some classes of graphs, and showed 
constant upper bounds for the ratio $\gamma(G)/\rho(G)$ in two cases. 
As we have seen in Section~\ref{sec:introduction}, L\"{o}wenstein~\cite{LowensteinRR13} 
showed that $\gamma(G)/\rho(G) \leq 2$ when $G$ is a cactus graph. 
We took a first step towards extending this result by studying maximal outerplanar 
graphs. We show that, in this class, the inequality $\gamma(G)/\rho(G) \leq 3$ holds.  
We believe that our technique to analyze maximal outerplanar graphs can be improved 
to obtain a better upper bound for this class of graphs.
\begin{conjecture}\label{conj:Outerplanar-2}
	For any maximal outerplanar graph~$G$, $\gamma(G) \leq 2\rho(G)$.
\end{conjecture}
If Conjecture~\ref{conj:Outerplanar-2} is true, then this upper bound is tight. 
Since the graph in Figure~\ref{fig:sun}, known as the sun graph, meets this bound. 
Further subclasses of planar graphs can be studied,
as planar triangulations and triangulated disks.
Finally, we considered the class of biconvex graphs and obtained a tight upper bound 
of $2$ for the ratio $\gamma(G)/\rho(G)$. We observe that for many classes of graphs 
it is unkwown whether they admit a constant upper bound or not. We depict this information 
in Figure~\ref{fig:constant-ub} for some classes of graphs. 

\begin{figure}
	\centering
	\includegraphics[scale=.3]{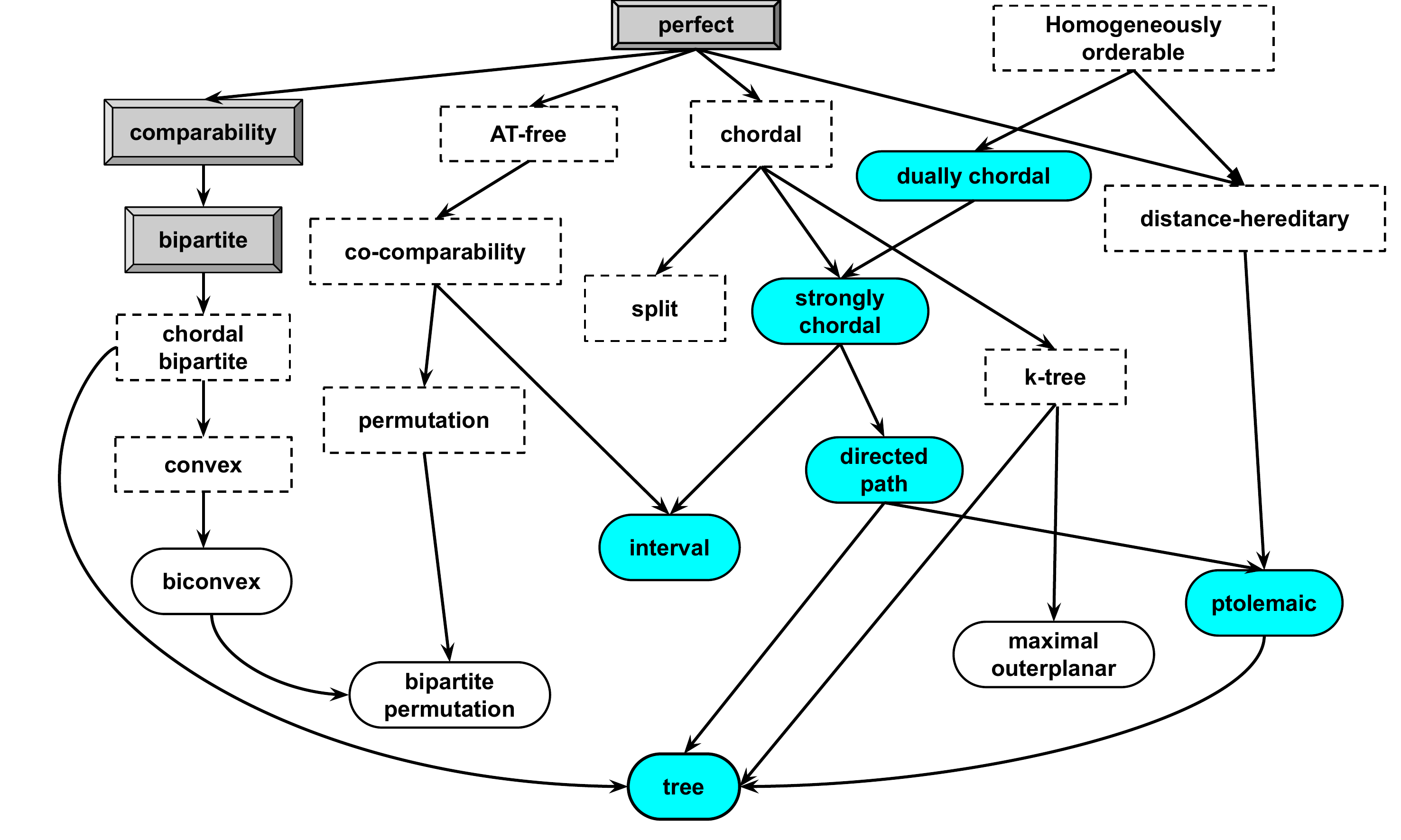}
	\caption{Constant upper bounds for the ratio $\gamma(G) / \rho(G)$ on some classes of graphs. Shaded ellipses 
	represent classes where $\gamma(G) \leq \rho(G)$. Clear ellipses indicate the results of our paper. That is, $\gamma(G) \leq 2\rho(G)$
	or $\gamma(G) \leq 3\rho(G)$. Moreover, 
	a shaded rectangle represents a class where the ratio $\gamma(G)/\rho(G)$ is unbounded, and a clear rectangle with 
	dashed border indicates that this relation is not known.}
	\label{fig:constant-ub}	
\end{figure}

\bibliographystyle{amsplain}
\bibliography{bibliography}



\end{document}